\documentclass[12pt]{amsart}
\usepackage{amssymb,amsfonts,amsmath,amsopn,amstext,amscd,latexsym,amsthm,enumerate,mathrsfs,bbm,color}

\usepackage[margin=36mm]{geometry}
\usepackage[all,cmtip]{xy}
\usepackage{longtable}

\newtheorem{theorem}{Theorem}[section]
\newtheorem{thm}[theorem]{Theorem}

\newtheorem{lem}[theorem]{Lemma}

\newtheorem{cor}[theorem]{Corollary}

\newtheorem{prob}[theorem]{Problem}

\theoremstyle{remark}

\numberwithin{equation}{section}

\begin{document}
\title[Some Remarks on Super $M_{p}$-groups]
{Some Remarks on Super $M_{p}$-groups}
\author[Xiaoyou Chen] {Xiaoyou Chen}
\address{School of Mathematics and Statistics, Henan University of Technology, Zhengzhou 450001, China}
\email{cxymathematics@hotmail.com}
\author[Ali Reza Moghaddamfar]{Ali Reza Moghaddamfar}
\address{Faculty of Mathematics, K. N. Toosi University of Technology, P. O. Box 16765--3381,
Tehran, Iran}
\email{moghadam@kntu.ac.ir}


\date{\today}


\begin{abstract}
Let $G$ be a finite group and $p$ be a prime divisor of $|G|$.
An irreducible $p$-Brauer character $\varphi$ of $G$ is called super-monomial if
every primitive $p$-Brauer character inducing $\varphi$ is linear. The group 
$G$ is said to be a super $M_{p}$-group if every irreducible $p$-Brauer character of $G$ is super-monomial.
In this note, we investigate the conditions under which a finite group $G$ qualifies as a super $M_{p}$-group.
We demonstrate that every normal subgroup of a super $M_{p}$-group of odd order is an $M_{p}$-group.
\end{abstract}

\maketitle

\section{Introduction}\label{sec1}
All groups considered in this note are assumed to be finite.
We refer the reader to \cite{Isaacs1976},
\cite{CSG}  and \cite{Navarro1998} for notation and terminology.
Recall that a character $\chi$ of a finite group $G$ is called {\em monomial} if
$\chi$ is induced from a linear character of some subgroup of $G$.
If all irreducible complex characters of $G$ are monomial,
then $G$ is referred to as an {\em $M$-group}.

A  fundamental result by Taketa \cite{Taketa} asserts that $M$-groups are necessarily solvable.
A well-known problem regarding $M$-groups is 
whether a normal subgroup of an $M$-group is always also an $M$-group.
Dade constructed an example in \cite{Dade} that demonstrates 
a normal subgroup of an $M$-group does not have to be an $M$-group.
However, Loukaki \cite{Loukaki} proved that if an $M$-group $G$ has order $p^{a}q^{b}$,
where $p$ and $q$ are odd primes, then every normal subgroup of $G$ remains an $M$-group.
Lewis further simplified Loukaki's proof in \cite{Lewis2006}.

We denote the set of all irreducible characters of $G$ by ${\rm Irr}(G)$.
An irreducible character $\chi\in {\rm Irr}(G)$ is called {\em super-monomial}
if every primitive character inducing $\chi$ is linear.
It is important to note that a character $\chi\in {\rm Irr}(G)$ is super-monomial
if and only if every character inducing $\chi$ is monomial.
We define a group $G$ as a {\em super $M$-group} if every irreducible character $\chi\in {\rm Irr}(G)$ is super-monomial.
Consequently, it follows that a super $M$-group must also be an $M$-group.

Isaacs  \cite{Isaacs1987} conjectured that every $M$-group of odd order is a super $M$-group.
Lewis  \cite{LewisProc2010} made progress on the topic of super $M$-groups by 
proving that Isaacs' conjecture holds true under certain conditions. Furthermore, 
if Isaacs' conjecture is true, then every normal subgroup of
an odd order $M$-group is also an $M$-group.

Let $p$ be a prime. Write ${\rm IBr}(G)$ for the set of irreducible
$p$-Brauer characters of a finite group $G$. A $p$-Brauer character
of $G$ is {\em monomial} if it is induced from a linear $p$-Brauer
character of some subgroup (not necessarily proper) of $G$. This
definition was introduced by Okuyama \cite{Okuyama} using Module
Theory. A group $G$ is called an {\em $M_{p}$-group} if every
irreducible $p$-Brauer character of $G$ is monomial. In particular,
Okuyama \cite{Okuyama} proved that $M_{p}$-groups are also solvable.
It is known by Fong-Swan theorem that an $M$-group is necessarily an $M_{p}$-group for
every prime $p$, however, an $M_{p}$-group may not be an $M$-group;
for example, ${\rm GL}(2, 3)$ is an $M_{2}$-group, not an $M$-group.
Bessenrodt \cite{Bessenrodt1990} proved that a normal Hall subgroup
of an $M_{p}$-group is also an $M_{p}$-group.

Similar to the definition of a super-monomial character,
we say that a $p$-Brauer character $\varphi\in {\rm IBr}(G)$ is {\em super-monomial} if
every primitive $p$-Brauer character inducing $\varphi$ is linear.
(An irreducible $p$-Brauer character is said to be primitive if it cannot be induced
by a $p$-Brauer character of a proper subgroup.)
A group $G$ is called a {\em super $M_{p}$-group} if every irreducible $p$-Brauer character of $G$
is super-monomial.

In this note, we discuss the conditions under which a finite group $G$ will be a super $M_{p}$-group.

\begin{thm}\label{theorem1}
Let $G$ be a group and $p$ be a prime divisor of $|G|$.

{\rm (i)} If all the primitive Brauer characters are linear and
every proper subgroup of $G$ is an $M_{p}$-group,
then $G$ is a super $M_{p}$-group.

{\rm (ii)} If $G$ is a super $M$-group,
then $G$ is a super $M_{p}$-group for every prime $p$.
\end{thm}

Observe that a super $M_{p}$-group is not necessarily a super $M$-group.
For example, ${\rm SL}(2, 3)$ is a super $M_{2}$-group but not an $M$-group.

Similar to Isaacs' conjecture for $M$-groups of odd order,
we propose the following problem.

\begin{prob}\label{problem1}
Let $G$ be a group and $p$ be a prime.
If $G$ is an $M_{p}$-group of odd order,
is it true that $G$ is a super $M_{p}$-group?
\end{prob}

It is important to note that if the prime $p$ does not divide the order of the group $|G|$,
then ${\rm IBr}(G)={\rm Irr}(G)$.
Consequently, if Problem \ref{problem1} has an affirmative answer, this would confirm 
Isaacs' conjecture.

Additionally, we have the following theorem.

\begin{thm}\label{theorem2}
Assume that every $M_{p}$-group with an odd order is a super $M_{p}$-group.
Then every normal subgroup of an odd-order $M_{p}$-group is also an $M_{p}$-group.
\end{thm}

In particular, if $p$ does not divide the order of $G$,
then ${\rm IBr}(G)={\rm Irr}(G)$, and this leads us to the following corollary, which aligns with 
 \cite[Corollary 2.2]{LewisProc2010}.

\begin{cor}\label{cor1}
If Isaacs' conjecture is true,
then normal subgroups of odd-order $M$-groups are also $M$-groups.
\end{cor}

\section{Proofs}\label{sec2}
We begin by presenting a lemma that will be utilized in the proof of Theorem \ref{theorem1}. 
Before proceeding, we need to introduce some additional notation. 
Let $G$ be a group, $p$ a prime, and $\chi\in {\rm Irr}(G)$ represent an irreducible character of $G$.
We denote $\chi^{0}$ as the restriction of $\chi$ to the set of $p$-regular elements
of $G$ (elements whose orders are not divisible by $p$). Furthermore, we say that  
$\chi$ is primitive if $\chi\neq \theta^{G}$ for any character $\theta$
of a proper subgroup of $G$.

According to the Fong-Swan theorem, if $\varphi$ is an irreducible $p$-Brauer
character of a $p$-solvable group $G$, then there exists a character $\chi\in {\rm Irr}(G)$
such that $\chi^{0}=\varphi$.

Based on the notation introduced above, we present the following lemma.

\begin{lem}\label{lemma2.1}
Let $G$ be a $p$-solvable group.
Suppose $\chi\in {\rm Irr}(G)$ and $\chi^{0}=\varphi\in {\rm IBr}(G)$.
If $\varphi$ is primitive, then $\chi$ is also primitive.
\end{lem}

\begin{proof}
This result follows from taking $\pi=p'$ in \cite[Lemma 3.3]{LewisProc2010},
where $p'$ is the complement of $p$ in the set of all primes.
\end{proof}

Next, we provide the proof of Theorem \ref{theorem1}.

\begin{proof}[Proof of Theorem \ref{theorem1}]
(i) Let $\varphi\in {\rm IBr}(G)$. If $\varphi$ is primitive,
then by hypothesis, $\varphi$ is linear and therefore super-monomial.

Now, assume that $\varphi$ is not primitive. In this case, there exists a proper subgroup $H$ of $G$
and some $\lambda\in {\rm IBr}(H)$ such that $\lambda^{G}=\varphi$.
Since $H$ is an $M_{p}$-group, it follows that $\lambda$ is monomial.
Consequently, $\varphi$ is super-monomial, which implies that $G$ is a super $M_{p}$-group.

(ii) Let $\varphi\in {\rm IBr}(G)$ and suppose $\varphi=\lambda^{G}$,
where $\lambda$ is a primitive $p$-Brauer character of a subgroup $H$ of $G$.
Note that $H$ is solvable since $G$ is solvable. Therefore, there exists $\mu\in {\rm Irr}(H)$ such that
$\mu^{0}=\lambda$. Since $\lambda$ is primitive, according to Lemma \ref{lemma2.1}, 
$\mu$ also be primitive.
Additionally, we have:
$$\varphi=\lambda^{G}=(\mu^{0})^{G}=(\mu^{G})^{0}.$$
It is important to observe that because $\varphi$ is irreducible, $\mu^{G}$ must be an irreducible character of $G$.
Since $G$ is a super $M$-group, it follows that $\mu$ is linear, which implies that $\lambda$ is also linear.
Consequently, $\varphi$ is super-monomial, and thus $G$ is a super $M_{p}$-group.
\end{proof}

Next, we will outline several lemmas that will be utilized in the proof of Theorem \ref{theorem2}. Let $\chi$ be a character of a group $G$.
Let $H$ be a subgroup of $G$, and let $\eta$ be a character of $H$.
We denote the restriction of $\chi$ to $H$ as $\chi_H$, and the induced character from $\eta$ as $\eta^G$.

The following lemma is known as  Mackey's theorem,
which can found in \cite[Problem 8.5]{Navarro1998}.

\begin{lem}[G. W. Mackey]\label{lem2.2}
Let $H$ and $K$ be subgroups of $G$, and
let $T$ be a set of double coset representatives relative to $H$ and $K$, such that
$$G=\bigcup_{t\in T}HtK,$$
is a disjoint union. Suppose $\varphi$ is a Brauer character of
$H$. For any element $g\in G$, we denote by $\varphi^{g}$ the Brauer character of
$H^{g}$ defined by $$\varphi^{g}(h^{g})=\varphi(h) \ \ \mbox{for
every} \ h\in H.$$ Then, we have 
$$(\varphi^{G})_{K}=\sum_{t\in T}(\varphi^{t}_{H^{t}\cap K})^{K}.$$
In particular, if $G=HK$, then $(\varphi^{G})_{K}=(\varphi_{H\cap
K})^{K}$. Furthermore, if $(\varphi^{G})_{K}$ is irreducible, then it follows that $G=HK$.
\end{lem}

\begin{lem}\label{lemma2.3}
Let $\varphi$ be a monomial Brauer character of a group $G$, and let $K\leq G$ be a subgroup of $G$.
If $\varphi_{K}$ is irreducible, then $\varphi_{K}$ is also a monomial Brauer character.
\end{lem}

\begin{proof}
Since $\varphi$ is a monomial Brauer character, there exists a subgroup $H$ of $G$ and a linear Brauer character
$\lambda\in {\rm IBr}(H)$ such that $\lambda^{G}=\varphi$.
By assumption, $(\lambda^{G})_{K}\in {\rm IBr}(K)$, and it follows
from Lemma \ref{lem2.2} that $G=HK$. Thus, we have:
$$\varphi_{K}=(\lambda^{G})_{K}=(\lambda_{H\cap K})^{K},$$
which demonstrates that $\varphi_{K}$ is indeed a monomial Brauer character.
\end{proof}

\begin{lem}\label{lemma2.4}
Let $\mathcal{G}$ be a class of groups that is closed under normal subgroups,
and let $\mathcal{M}$ be the class of $M_{p}$-groups in $\mathcal{G}$.
Suppose that every group in $\mathcal{M}$ is a super $M_{p}$-group.
Then, $\mathcal{M}$ is closed under normal subgroups.
\end{lem}

It is important to note that $\mathcal{G}$ is a class of groups that is closed under normal subgroups.
This means that if $G\in \mathcal{G}$ and $N\lhd G$, then $N\in \mathcal{G}$.
Additionally, since $\mathcal{M}$ consists of $M_{p}$-groups in $\mathcal{G}$, 
this implies that if $G\in \mathcal{M}$ and $N\lhd G$, then $N\in \mathcal{G}$. Therefore, 
if we can prove that $\mathcal{M}$ is closed under normal subgroups, we can conclude that 
$N\in \mathcal{M}$.

\begin{proof}[Proof of Lemma \ref{lemma2.4}]
Fix a group $G\in \mathcal{M}$, and let $N$ be a normal subgroup of $G$.
If $N=G$, the result is trivial. Therefore, we assume that $N<G$,
and there exists a maximal normal subgroup $M$ of $G$ such that $N\leq M<G$.
Since $G$ is an $M_{p}$-group it follows that $G$ is solvable.
Thus, $|G: M|=r$ for some prime $r$.
Notice that $M\in \mathcal{G}$ because $\mathcal{G}$ is closed under normal subgroups.

Let $\theta$ be an irreducible $p$-Brauer character of $M$. Define the inertia group of $\theta$ in $G$ as
$$T={\rm I}_{G}(\theta)=\{g\in G\mid \theta^{g}=\theta\}.$$ 
Since $|G: M|$ is a prime, we conclude that either $T=G$ or $T=M$.
By standard Clifford theory, it follows that either $\theta$ extends to $G$
or $\theta$ induces irreducibly to $G$.

If $\theta$ extends to $G$,
there exists a Brauer character $\varphi\in {\rm IBr}(G)$ such that $\varphi_{M}=\theta$.
Since $G$ is an $M_{p}$-group, we can apply Lemma \ref{lemma2.3} to conclude that $\theta$ is monomial.
If $\theta$ induces irreducibly to $G$, then $\theta^{G}\in {\rm IBr}(G)$.
Given that $G$ is a super $M_{p}$-group, we also deduce that $\theta$ is monomial.
Thus, we conclude that $M$ is an $M_{p}$-group, which implies that $M\in \mathcal{M}$.

Notice that $|M: N|<|G: N|$.
By induction on $|G: N|$ and the previous proof, we find that $N\in \mathcal{M}$.
Therefore, we can state that $\mathcal{M}$ is closed under normal subgroups.
\end{proof}

\begin{proof}[Proof of Theorem \ref{theorem2}]
We note that the class of groups with odd order is closed under normal subgroups
because normal subgroups of an odd order group also have odd order.
Therefore, the theorem follows from Lemma \ref{lemma2.4}.
\end{proof}

\section*{Acknowledgments}
The authors are very much indebted to the referee for his or her
invaluable comments which polished this paper a lot.
The authors also thank support of the program of Henan University of
Technology (2024PYJH019), the Natural Science Foundation of Henan
Province (252300421983), and the American Mathematical Society's Ky
and Yu-Fen Fan fund(253H2023LC253).

\end{document}